\documentclass[12pt]{amsart}
\usepackage{amssymb}
\usepackage{amscd}

\newtheorem{theorem}{Theorem}[section]
\newtheorem{corollary}[theorem]{Corollary}
\newtheorem{lemme}[theorem]{Lemma}
\newtheorem{defin}{Definition}[section]

\allowdisplaybreaks \theoremstyle{definition}
\theoremstyle{remark} \numberwithin{equation}{section}

\begin{document}

\title [Beurling's theorem for the Clifford-Fourier transform]{Beurling's theorem for  the  Clifford-Fourier transform\\}%
\author[ R. Jday, J. El Kamel]{ Rim Jday \qquad and \qquad Jamel El Kamel }
 \address{Jamel El Kamel. University of Monastir, Faculty of Sciences Monastir, 5000 Tunisia.}
 \email{jamel.elkamel@fsm.rnu.tn}
 \address{Rim Jday . University of Tunis El Manar, Faculty of Sciences Tunis, 2092 Tunis, Tunisia.}
 \email{rimjday@live.fr}
 \begin{abstract} We give a generalization of Beurling's theorem for the Clifford-Fourier transform. Then, analogues of Hardy, Cowling-Price and  Gelfand-Shilov  theorems are obtained in Clifford analysis. 
\end{abstract}
\maketitle
{\it Keywords:} Clifford analysis; Clifford-Fourier transform; Uncertainty principles; Beurling's theorem.\\

\noindent MSC (2010): 42B10, 30G35.\\
\section{Introduction}
Uncertainty principle asserts that a function  and its Fourier transform cannot both be sharply localized. In Euclidien spaces, 
many theorems are devoted  to clarify it  such as Beurling, Cowling and price, Hardy, Heisenberg..\\
 Beurling theorem   which is given by A. beurling \cite{key-1} and proved by H$\rm{\ddot{o}}$rmander \cite{key-2} is the the most relevant one: that it gives Hardy, Cowling-Price and Gelfand-Shilov theroems.
\begin{theorem}
 Let $f\in L^2(\mathbb{R})$  be such that
 $$\displaystyle\int_{\mathbb{R}}\int_{\mathbb{R}}|f(x)||\widehat{f}(y)|e^{|x||y|}dxdy <\infty,$$
  then $f=0$.
  \end{theorem}
This theorem is generalized by Bonami et al \cite{key-3} by giving solutions in terms of Hermite functions.\\
\begin{theorem}Let $N\geq 0$. Assume $f\in L^2(\mathbb{R}^m)$ satisfying 
 $$\displaystyle\int_{\mathbb{R}^m}\int_{\mathbb{R}^m}\frac{|f(x)||\widehat{f}(y)|}{(1+|x|+|y|)^N}e^{|x||y|}dxdy <\infty,$$
  then $f(x)=P(x)e^{-a|x|^2}$ where $P$ is a polynomial of degree $<\frac{ N-m}{2}$ and $a>0$.
\end{theorem}
\noindent In 2010, Kawazoe and Majjaoli provide Beurling theorem for the Dunkl transform \cite{key-4}. Moreover, Parui and Pusti give an alternative proof similar of that one used in \cite{key-3} for the Dunkl transform (see \cite{key-5}).\\
Our aim is to establish Beurling theorem for the Clifford-Fourier transform given by Brackx et al \cite{key-6} and studied in \cite{key-7,key-8}.\\
This paper is organized as follows. In section 2, we recall  Clifford algebra and some notations that will be usefull in the sequel. In section 3, we remind the Clifford-Fourier transform and its properties. In section 4, we prove Beurling theorem for the Clifford-Fourier transform. Section 5 contains other uncertainty principles in Clifford analysis : Hardy, Cowling and Price and Gelfand Shilov.

\section{Notations and preliminaries}
 Clifford algebra  $Cl_{0,m}$  over $\mathbb{R}^m$ is defined as an algebra  generated by the $2^m$-dimensional basis:
\begin{equation}
\{1,e_1,e_2,e_3,..,e_m,e_{12},...,e_{12..m} \},
\end{equation} 
where the multiplication of vectors from this  basis is governed by the  rules:

\begin{equation}
 \left\{
\begin{array}{cc}
 \displaystyle  e_ie_j=-e_je_i,\qquad\qquad \text{if}\quad i\neq j; \qquad\\
   \displaystyle e_i^2=-1,\qquad\qquad\quad \forall 1\leq i\leq m. \\

\end{array}
\right.
 \end{equation} 
  Clifford algebra $Cl_{0,m}$ is decomposed as :
 \begin{equation}
 Cl_{0,m}=\oplus_{k=0}^m Cl_{0,m}^k,
  \end{equation}
 where 
 $Cl_{0,m}^k=span\{e_{i_1}...e_{i_k}, i_1<..<i_k\}$\\
 A multivector $x$ on the Clifford algebra can be presented by:
  \begin{equation}
 x=\displaystyle\sum_{A\in J}x_Ae_A,
  \end{equation} 
 where $J :=\{0,1,..,m,12,..,12..m\},$ $x_A$ real number  and $e_A$ belongs to the  basis of $Cl_{0,m}$ defined above.\\
For each multivector $x$, the Clifford  norm is:
  \begin{equation}
 {||x||}_c={\left(\displaystyle\sum_{A\in J} x_A^2\right)}^{\frac{1}{2}}.
  \end{equation}
Thus, a  vector  $x$  in $Cl_{0,m}$ can be identify with 
\begin{equation}
 x=\displaystyle\sum_{i=1}^mx_ie_i,
\end{equation}
and it's norm is 
\begin{equation}
  ||x||_c^2=\displaystyle\sum_{i=1}^{m}x_i^2.
 \end{equation}
 
 We introduce the   Dirac operator,   Gamma operator and Laplace operator associated to a vector $x$ respectively by:
  \begin{equation}
  \partial_{x}=\displaystyle\sum_{i=1}^m e_i \partial_{x_i};
  \end{equation}
   
 \begin{equation}
\qquad\qquad\quad\qquad\Gamma_x=-\displaystyle\sum_{j<k}e_je_k(x_j\partial_{x_k}-x_k\partial_{x_j});
  \end{equation}
 \begin{equation}
   \Delta_c=\displaystyle\sum_{i=1}^m \partial_i^2.\quad
\end{equation}
    In the Clifford algebra, a vector x  and the Dirac operator satisfies :
 \begin{equation}
   ||x||_c^2=-x^2
 \end{equation}
 and
 \begin{equation}
\Delta_c=-\partial_x^2.
 \end{equation}
The inner product and the wedge product of two vectors $ x$ and $y$ are given respectively by:
  \begin{equation}
 \quad <x,y> :=\displaystyle\sum_{j=1}^m x_jy_j=\frac{-1}{2}(xy+yx);
  \end{equation}
   \begin{equation}
\qquad \qquad \qquad\quad  x\wedge y:=\displaystyle\sum_{j<k}e_je_k(x_j y_k-x_ky_j)=\frac{1}{2}(xy-yx).
 \end{equation}

 \noindent Every function $f: \mathbb{R}^m \rightarrow Cl_{0,m}$ can be written as :
  \begin{equation}
 f(x)=f_0(x)+\displaystyle\sum_{i=1}^me_if_i(x)+\displaystyle\sum_{i<j}e_ie_jf_{ij}(x)+..+e_1..e_mf_{1..m}(x),
  \end{equation}
 where  $f_0,f_i,..,f_{1..m}$ all real-valued functions.\\
 We denote by :\\ 
 $\bullet \mathcal{P}_k$ the space of homogeneous polynomials of degree $k$ taking values in $Cl_{0,m},$\\
 $\bullet \mathcal{P}$ the space of polynomials taking values in $Cl_{0,m}$, i.e\\
$$ P := \mathbb{R}[x_1, . . . , x_m] \otimes Cl_{0,m}.$$
$\bullet \mathcal{M}_k:=Ker \partial_{x}\cap \mathcal{P}_k$ the space of spherical monogenics of degree $k,$\\
$\bullet B(\mathbb{R}^m)\otimes Cl_{0,m}$ a class of integrable functions taking values in $Cl_{0,m}$ and satisfying
 \begin{equation}
||f||_B:=\int_{\mathbb{R}^m}(1+{||y||}_c)^\frac{m-2}{2}||f(y)||_cdy< \infty,
 \end{equation}
$\bullet L^p(\mathbb{R}^m)\otimes Cl_{0,m}$ the space of integrable functions  taking values in $Cl_{0,m}$ such that 
 \begin{equation}
||f||_{p,c}=\left(\displaystyle\int_{\mathbb{R}^m}{||f(x)||}_c^pdx\right)^{\frac{1}{p}}
=\left(\displaystyle\int_{\mathbb{R}^m}\left(\displaystyle\sum_{A\in J}(f_A(x))^2\right)^\frac{p}{2}dx\right)^{\frac{1}{p}}<\infty,
 \end{equation}
where $J=\{0,1,..,m,12,13,23..,12..m\},$\\
$\bullet \mathcal{S}(\mathbb{R}^{m})$ the Schwartz space of infinitely differentiable functions on $\mathbb{R}^{m}$ which are rapidly decreasing as their derivatives.

 \section{Clifford-Fourier Transform}
 \begin{defin} \cite{key-8} The Clifford-Fourier transform is defined on  $B(\mathbb{R}^{m})\otimes Cl_{0,m}$  by  
  \begin{equation}
 \mathcal{F}_{\pm}(f)(y)={(2\pi)}^\frac{-m}{2}\int_{\mathbb{R}^m} K_{\pm}(x,y)f(x)dx,
  \end{equation}
 where 
  \begin{equation}
 K_{\pm}(x,y)=e^{\mp i\frac{\pi}{2}\Gamma_y}e^{-i<x,y>}.
  \end{equation}
 \end{defin}
 
 \begin{lemme} \cite{key-9} Let  $m$ be even.  Then 
  \begin{equation}
{||K_{\pm}(x,y)||}_c\leq C e^{{||x||}_c{||y||}_c}, \qquad \forall x,y\in \mathbb{R}^m,
 \end{equation}
\end{lemme}
\begin{theorem} Let m be even  and $f\in B(\mathbb{R}^m)\otimes Cl_{0,m}$. Then, there exists a positive constant A such that
\begin{equation}
||\mathcal{F}_{\pm}(f)(y)||_c \leq C e^{\frac{||y||_c^2}{4}}||f||_B,\quad \forall ||y||_c>A
\end{equation}
\begin{equation}
 \qquad \quad||\mathcal{F}_{\pm}(f)(y)||_c \leq C (1+A)^{\frac{m-2}{2}}||f||_B, \quad\forall ||y||_c\leq A
\end{equation}
\end{theorem}
\begin{proof} Let $k(y)=(1+||y||_c)^{\frac{m-2}{2}}e^{\frac{-||y||_c^2}{4}},\quad \forall y\in \mathbb{R}^m$.\\
\noindent Since $\lim_{||y||_c\rightarrow \infty}k(y)=0$, there exists $A>0$ such that  for all $||y||_c>A$
$$||k(y)||_c\leq 1.$$
Thus, it follows that 
$$(1+||y||_c)^{\frac{m-2}{2}}\leq e^{\frac{||y||_c^2}{4}}, \qquad \forall ||y||_c>A.$$
Recall that the Clifford kernel \cite{key-8}  is written as :
$$K_-(x,y)=K_0^-(x,y)+\displaystyle\sum_{i<j} e_{ij}K_{ij}^-(x,y),$$
where $K_{0}^-(x,y)$ and $K_{ij}^-(x,y)$ satisfyies
$$|K_0^-(x,y)|\leq C (1+||x||_c)^{\frac{m-2}{2}}(1+||y||_c)^{\frac{m-2}{2}}$$
$$|K_{ij}^-(x,y)|\leq C (1+||x||_c)^{\frac{m-2}{2}}(1+||y||_c)^{\frac{m-2}{2}}.$$
We conclude.
\end{proof}
\begin{theorem}\cite{key-7}\\
\noindent  1) The Clifford-Fourier transform is a continuous operator from $\mathcal{S}(\mathbb{R}^m)\otimes Cl_{0,m}$ to $\mathcal{S}(\mathbb{R}^m)\otimes Cl_{0,m}$.\\
In particular,  when m  even,  we have 
$$\mathcal{F}_+\mathcal{F}_+=id_{\mathcal{S}(\mathbb{R}^m)\otimes Cl_{0,m}}.$$
2) The Clifford-Fourier transform extends from $\mathcal{S}(\mathbb{R}^m)\otimes Cl_{0,m}$ to a continuous map on $L^2(\mathbb{R}^m)\otimes Cl_{0,m}$.\\ In particular, when $m$  even, we have 
$$||\mathcal{F}_{\pm}(f)||_{2,c}=||f||_{2,c},$$
$\text{for all}\quad f\in L^2(\mathbb{R}^m)\otimes Cl_{0,m}$.
\end{theorem}
\begin{theorem} \cite{key-9} Let $a>0$ and $P\in \mathcal{P}_k(\mathbb{R}^m)$. Then, there exists $Q\in \mathcal{P}_k(\mathbb{R}^m)$  satisfying :
\begin{equation}
\mathcal{F}_{\pm}(P(.)e^{-a||.||_c^2})(x)=Q(x)e^{-\frac{||x||_c^2}{4a}}.
\end{equation}

\end{theorem}
\begin{defin}\cite{key-8} 
Let m be even. The Clifford translation  and the Clifford convolution for $f, g \in \mathcal{S}(\mathbb{R}^{m})$ are introduced respectively by 
\begin{equation}
\quad T_yf(x)={(2\pi)}^{-\frac{m}{2}}\int_{\mathbb{R}^m}\overline{K_-(\epsilon,x)}K_-(y,\epsilon)\mathcal{F}(f)(\epsilon)d\epsilon,
\end{equation}
\begin{equation}
f\ast_{Cl} g(x)={(2\pi)}^{-\frac{m}{2}}\int_{\mathbb{R}^m}T_yf(x) g(y) dy.\qquad\qquad
\end{equation}

\end{defin}
\begin{theorem}\cite{key-8} Let $f \in \mathcal{S}(\mathbb{R}^m)\otimes Cl_{0,m}$\\
i) For $m=2$,
$$T_yf(x)=f(x-y)$$
 ii) For m even and $m>2$, we have $$T_yf(x) = f_0(|x - y|),$$
  for radial function $f$ on $\mathbb{R}^m$, $f(x) = f_0(|x|)$ with
$f_0 : \mathbb{R}_+ \rightarrow \mathbb{R}$.

\end{theorem}

\begin{theorem}\cite{key-8} Let $f\in \mathcal{S}(\mathbb{R}^m)$ be radial function and $g\in \mathcal{S}(\mathbb{R}^m)\otimes Cl_{0,m}$. Then,
$$\mathcal{F}_{\pm}(f\ast_{Cl} g)=\mathcal{F}_{\pm}(f)\mathcal{F}_{\pm}(g).\quad$$
In particular, we have 
$$f\ast_{Cl} g=g\ast_{Cl} f$$
\end{theorem}

\section{Beurling's theorem for the Clifford-Fourier transform}
In this section, we provide Beurling's theorem for the Clifford-Fourier transform.
\begin{lemme}Assume $f\in L^2(\mathbb{R}^m)\otimes Cl_{0,m}$ satisfying  
\begin{equation} 
\displaystyle\int_{\mathbb{R}^m}\int_{\mathbb{R}^m}\frac{||f(x)||_c||\mathcal{F}_{\pm}(f)(y)||_c}{(1+||x||_c+||y||_c)^N}e^{||x||_c||y||_c}dxdy <\infty,
\end{equation}
for some $N\geq0$.\\
\noindent Then, $f\in B(\mathbb{R}^m)\otimes Cl_{0,m}$ and  $\mathcal{F}_{\pm}(f )\in B(\mathbb{R}^m)\otimes Cl_{0,m}$. 

\end{lemme}

\begin{proof}
We may suppose $f\neq0$.\\
\noindent Applying Fubini's theorem, we obtain for almost every $y\in \mathbb{R}^m$,
$$||\mathcal{F}_{\pm}(f)(y)||_c\displaystyle\int_{\mathbb{R}^m}\frac{||f(x)||_c}{(1+||x||_c+||y||)^N} e^{||x||_c||y||_c} dx<\infty.$$ 
Since $f\neq 0$, then $\mathcal{F}_{\pm}(f)\neq 0$. Thus, there exists $y_0\neq 0$ such that $\mathcal{F}_{\pm}(f)(y_0)\neq 0$ and
\begin{equation}
\displaystyle\int_{\mathbb{R}^m}\frac{||f(x)||_c}{(1+||x||_c)^N}  e^{||x||_c||y_0||_c}dx <\infty.
\end{equation}
Using  (4.2) and  the fact that for large $x$

$$\frac{ e^{||x||_c||y_0||_c}}{(1+||x||_c)^{N+\frac{m-2}{2}}}\geq 1,$$
 it follows   that 
 $$\displaystyle\int_{\mathbb{R}^m}(1+||x||_c)^{\frac{m-2}{2}}||f(x)||_c dx<\infty.$$
Similarly, we get $\mathcal{F}_{\pm}(f)\in B(\mathbb{R}^m)\otimes Cl_{0,m}$.

\end{proof}
\begin{theorem}\cite{key-11}
Let $\phi$ be an entire function of order $2$ in the complex plane and let $\alpha \in ]0,\displaystyle\frac{\pi}{2}[$.
Assume that $|\phi(z)|$ is bounded by $C(1 + |z|)^N$ on the boundary of some
angular sector $\{re^{i\beta} : r \geq 0, \beta_0\leq \beta \leq\beta_0+\alpha\}$. Then the same bound is
valid inside the angular sector (when replacing $C$ by $2^NC$).
\end{theorem}
\begin{theorem} Let $m$ be even and $N$ be a positive integer.
Assume $f\in L^2(\mathbb{R}^m)\otimes Cl_{0,m}$ such that 
\begin{equation} 
\displaystyle\int_{\mathbb{R}^m}\int_{\mathbb{R}^m}\frac{||f(x)||_c||\mathcal{F}_{\pm}(f)(y)||_c}{(1+||x||_c+||y||_c)^N}e^{||x||_c||y||_c}dxdy <\infty.
\end{equation}
Then, 
$$f(x)=Q(x)e^{-a||x||_c^2}.$$
for some $ a>0$ and polynomial $Q$ with degree less than $\displaystyle\frac{N-m}{2}$.
\end{theorem}
\begin{proof}
{\bf Step 1.}\\
Let $ g(x) =f\ast_{Cl} e^{-\frac{||.||_c^2}{2}}(x)$.\\
By lemma 4.1, we have $f\in B(\mathbb{R}^m)\otimes Cl_{0,m}$.
Thus,  $g\in B(\mathbb{R}^m)\otimes Cl_{0,m}$.\\
Theorem 3.6 and theorem 3.4 yield  

\begin{equation}\mathcal{F}_{\pm}(g)(x)=\mathcal{F}_{\pm}(f)(x)\mathcal{F}_{\pm}(e^{-\frac{||.||_c^2}{2}})(x)=\mathcal{F}_{\pm}(f)(x)e^{-\frac{||x||^2_c}{2}}.
\end{equation}
We will show that $g$ satisfies the following assumptions :\\
i) $$\displaystyle\int_{\mathbb{R}^m}||\mathcal{F}_{\pm}(g)(y)||_ce^{\frac{||y||_c^2}{2}}dy<\infty,$$
ii)$$||\mathcal{F}_{\pm}(g)(y)||_c\leq C e^{-\frac{||y||_c^2}{4}},$$
iii)  $$\displaystyle\int_{\mathbb{R}^m}\int_{\mathbb{R}^m}\frac{||g(x)||_c||\mathcal{F}_{\pm}(g)(x)||_ce^{||x||_c||y||_c}}{(1+||x||_c+||y||_c)^N}dxdy<+\infty$$
4i) $$\qquad\qquad\displaystyle\int_{||x||_c\leq R}\int_{\mathbb{R}^m}||g(x)||_c||\mathcal{F}_{\pm}(g)(y)||_ce^{||x||_c||y||_c}dxdy\leq C(1+R)^N.$$
Since $\mathcal{F}(f)\in B(\mathbb{R}^m)\otimes Cl_{0,m}$, i) is a simple deduction from (4.4).\\
\noindent Let's  prove ii).\\
\noindent Using (4.4) and theorem 3.2 , it  follows that 
$$\qquad\quad \qquad||\mathcal{F}_{\pm}(g)(y)||_c\leq C(1+A)^\frac{m-2}{2}||f||_Be^{-\frac{||y||_c^2}{2}} ,\quad\forall ||y||_c\leq A$$
and 
$$||\mathcal{F}_{\pm}(g)(y)||_c\leq C||f||_Be^{-\frac{||y||_c^2}{4}} ,\quad\forall ||y||_c>A.$$
Thus, we get ii)
where $C$ is constant depending on $f$.\\
In order to establish iii), we use (4.4), theorem 3.5 and theorem 3.7.\\
 Therefore, we find 
$$I:=\displaystyle\int_{\mathbb{R}^m}\int_{\mathbb{R}^m}\frac{||g(x)||_c||\mathcal{F}_{\pm}(g)(y)||_ce^{||x||_c||y||_c}}{(1+||x||_c+||y||_c)^N}dxdy $$
$$\qquad\qquad\qquad\qquad\quad\leq\displaystyle\int_{\mathbb{R}^m}\int_{\mathbb{R}^m}\displaystyle\int_{\mathbb{R}^m} \frac{||f(t)||_ce^{-\frac{||x-t||_c^2}{2}}||\mathcal{F}_{\pm}(f)(y)||_ce^{-\frac{||y||_c^2}{2}}e^{||x||_c||y||_c}}{(1+||x||_c+||y||_c)^N}dtdxdy$$
$$\qquad\quad\leq\displaystyle\int_{\mathbb{R}^m}\int_{\mathbb{R}^m}||f(t)||_c||\mathcal{F}_{\pm}(f)(y)||_cA(t,y)e^{||t||_c||y||_c}dtdy,$$
with $A(t,y):=e^{-\frac{||t||_c^2}{2}}e^{-\frac{||y||_c^2}{2}}e^{-||t||_c||y||_c}\displaystyle\int_{\mathbb{R}^m}\frac{e^{-\frac{||x||_c^2}{2}}e^{<x,t>}e^{||x||_c||y||_c}}{(1+||x||_c+||y||_c)^N}dx.$\\
\noindent We should  prove that 
\begin{equation}
A(t,y)\leq C( 1+||t||_c+||y||_c)^{-N}.
\end{equation}
According to Cauchy-Schwarz's inequality, we have
$$|<x,t>|\leq ||x||_c||t||_c.$$
Thus
$$\displaystyle\int_{\mathbb{R}^m}\frac{e^{-\frac{||x||_c^2}{2}}e^{<x,t>}e^{||x||_c||y||_c}}{(1+||x||_c+||y||_c)^N}dx\leq e^{\frac{(||t||_c+||y||_c)^2}{2}}\int_{\mathbb{R}^m}\frac{e^{-\frac{(||x||_c-||t||_c-||y||_c)^2}{2}}}{(1+||x||_c+||y||_c)^N}dx.$$
Hence 
$$A(t,y)\leq\displaystyle\int_{\mathbb{R}^m}\frac{e^{-\frac{(||x||_c-||t||_c-||y||_c)^2}{2}}}{(1+||x||_c+||y||_c)^N}dx.$$
Fix $0<c<1$.
Let $B=(1+||t||_c+||y||_c)$.
$$A(t,y)\leq \displaystyle\int_{|||x||_c-||t||_c-||y||_c|> cB}e^{-\frac{(||x||_c-||t||_c-||y||_c)^2}{2}}dx+\int_{|||x||_c-||t||_c-||y||_c|\leq cB}\frac{e^{-\frac{(||x||_c-||t||_c-||y||_c)^2}{2}}}{(1+||x||_c+||y||_c)^N}dx.$$
If $|||x||_c-||t||_c-||y||_c|\leq cB$, then  
$$1+||x||_c+||y||_c\geq 1+\frac{||t||_c}{2}-\frac{||||x||_c-||t||_c|}{2}+||y||_c$$ 
$$\qquad \qquad\qquad\qquad\qquad\quad\geq \frac{1}{2}+\frac{||t||_c}{2}+\frac{||y||_c}{2}-\frac{|||x||_c-||t||_c-||y||_c|}{2}$$
$$\geq \frac{(1-c)}{2}B.\qquad\quad$$
  The proof of (4.5) and  iii) is carried out by (4.3).\\
 4i) Fix $k>4$.
 $$J:=\displaystyle\int_{||x||_c\leq R}\int_{\mathbb{R}^m}||g(x)||_c||\mathcal{F}_{\pm}(g)(y)||_ce^{||x||_c||y||_c}dxdy$$
 $$=\displaystyle\int_{||x||_c\leq R}||g(x)||_c\left(\int_{||y||>kR}||\mathcal{F}_{\pm}(g)(y)||_ce^{||x||_c||y||_c}dy+\int_{||y||<kR}||\mathcal{F}_{\pm}(g)(y)||_ce^{||x||_c||y||_c}dy\right)dx.$$
 ii) implies  
$$\qquad\quad J\leq \displaystyle\int_{||x||_c\leq R}||g(x)||_c\left(\int_{||y||>kR}C e^{-(\frac{1}{4}-\frac{1}{k})||y||_c^2} dy+\int_{||y||<kR}||\mathcal{F}_{\pm}(g)(y)||_ce^{||x||_c||y||_c}dy\right)dx$$
$$\leq C ||g||_{1,c}+\int_{||x||_c\leq R}\int_{||y||<kR}||g(x)||_c||\mathcal{F}_{\pm}(g)(y)||_ce^{||x||_c||y||_c}dydx.\qquad$$
Multiplying and dividing by $(1+||x||_c+||y||_c)^N$ in the integral of right side, we obtain 
$$J\leq C(1+R)^N\int_{||x||_c\leq R}\int_{||y||<kR}\frac{||g(x)||_c||\mathcal{F}_{\pm}(g)(y)||_ce^{||x||_c||y||_c}}{(1+||x||_c+||y||_c)^N}dxdy.$$
iii) completes the proof of 4i).\\
{\bf Step 2.}
Combining theorem 3.3, lemma 3.1 and ii),   we get  $g$ admits an holomorphic extension to $\mathbb{C}\otimes \mathbb{R}^m$. Moreover, for all $z\in\mathbb{C}\otimes\mathbb{R}^m$
$$||g(z)||_c=||\mathcal{F}_\pm\circ \mathcal{F}_\pm(g)(z)||_c\qquad\qquad$$
$$\qquad\qquad\qquad\qquad\leq(2\pi)^{-\frac{m}{2}}\displaystyle\int_{\mathbb{R}^m}e^{||y||_c||z||_c}||\mathcal{F}_\pm(g)(y)||_cdy$$
$$\qquad\qquad\qquad\quad\leq (2\pi)^{-\frac{m}{2}}C\displaystyle\int_{\mathbb{R}^m}e^{||y||_c||z||_c} e^{-\frac{||y||_c^2}{4}}dy$$
$$\leq Ce^{||z||_c^2}.\qquad\quad$$
Thus, $g$ is entire of order $2$.\\
We should prove that for all $\mathbb{C}\otimes \mathbb{R}^m$, $g(z)g(iz)$  is a polynomial.\\
Using theorem 3.3 and lemma 3.1, it follows that for all $x\in\mathbb{R}^m$ and $\theta \in\mathbb{R}$ 
\begin{equation}||g(e^{i\theta}x)||_c\leq C \displaystyle\int_{\mathbb{R}^m}e^{||x||_c||y||_c}||\mathcal{F}_\pm(g)(y)||_cdy.
\end{equation}
Let
 $$\Gamma(z)=\displaystyle\int_0^{z_1}...\int_0^{z_m}g(u)g(iu)du.$$
Following the proof of \cite{key-3}, we conclude.
\end{proof}
\section{Applications to other uncertainty principles }
In this section, we show the relevance of Beurling theorem since it entail Hardy, cowling- price and Gelfand Shilov theorems. 
\begin{corollary}[Hardy theorem] Let  $m$ be even.
Assume that  $f\in L^2(\mathbb{R}^m)\otimes Cl_{0,m}$ satisfies
\begin{equation}
||f(x)||_c\leq C (1+||x||_c)^N e^{-a||x||_c^2}
\end{equation}
and 
\begin{equation}
||\mathcal{F}(f)(y)||_c \leq C (1+||y||_c)^Ne^{-b ||y||_c^2},\quad
\end{equation}
  for some $N\in \mathbb{N}$.
\noindent Then, three cases can occur\\ 
i) If $ab>\frac{1}{4}$, then $f=0$.\\
ii)  If $ab=\frac{1}{4}$, then $f=P(x)e^{-a||x||_c^2}$ whith  $deg P\leq N$.\\
ii) If $ab<\frac{1}{4}$, there are many functions satisfying these estimates.
\end{corollary}

\begin{corollary}[Cowling and Price theorem] Let $N\in \mathbb{N}$, $1<p,q<\infty$  and $m$ be even.
Let  $f\in L^2(\mathbb{R}^m)\otimes Cl_{0,m}$ be  such that
\begin{equation}
\displaystyle\int_{\mathbb{R}^m}\frac{e^{\alpha p||x||_c^2}||f(x)||_c^p}{(1+||x||_c)^N}dx<\infty
\end{equation}
\begin{equation}
\displaystyle\int_{\mathbb{R}^m}\frac{e^{\beta q ||y||_c^2}||\mathcal{F}(f)(y)||_c^q }{(1+||y||_c)^N}dy<\infty
\end{equation}
with $\alpha \beta\geq\frac{1}{4}$  and  $\frac{1}{p}+\frac{1}{q}=1$. 
Then\\ 
\noindent i) $f=0$,  if $\alpha \beta>\frac{1}{4}$.\\
ii) $f=P(x)e^{-\alpha||x||_c^2}$ whith  $P$ is a polynomial of degree  $<min\{\frac{N-m}{p},\frac{N-m}{q}\}$, if $\alpha \beta=\frac{1}{4}$.
\end{corollary}

\begin{corollary}[Gelfand Shilov theorem] Let $N\in \mathbb{N}$, $1<p,q<\infty$, $\alpha , \beta >0$  and $m$ be even.
Let  $f\in L^2(\mathbb{R}^m)\otimes Cl_{0,m}$  satisfy 
\begin{equation}
\displaystyle\int_{\mathbb{R}^m}\frac{||f(x)||_ce^{\frac{(2\alpha||x||_c)^p}{p}}}{(1+||x||_c)^N}dx<\infty
\end{equation}
\begin{equation}
\displaystyle\int_{\mathbb{R}^m}\frac{||\mathcal{F}(f)(y)||_ce^{\frac{(2\beta ||y||_c)^q}{q}}}{(1+||y||_c)^N}dy<\infty
\end{equation}
with $\alpha \beta\geq\frac{1}{4}$  and  $\frac{1}{p}+\frac{1}{q}=1$. We have the following results :\\ 
i) If $\alpha \beta>\frac{1}{4}$ or $(p,q)\neq (2,2)$, then $f=0$.\\
ii) If $\alpha \beta=\frac{1}{4}$ and $p=q=2$, then $f=P(x)e^{- 2 \alpha^2||x||_c^2}$ where $P$ is a polynomial with degree less than $N-m$.
\end{corollary}

\end{document}